\documentclass[reqno]{amsart}

\usepackage[dvipsnames]{xcolor}

\usepackage{iftex}
\ifPDFTeX 
  \usepackage[utf8]{inputenc}
  \usepackage[T1]{fontenc}
\else 
  \usepackage{fontspec}
\fi
\usepackage{lmodern}

\usepackage{geometry}
\usepackage{mathtools}
\usepackage{yfonts}
\usepackage[foot]{amsaddr}

\usepackage[pdfencoding=auto]{hyperref}
\hypersetup{
    colorlinks,
    citecolor=LimeGreen,
    filecolor=black,
    linkcolor=blue,
    urlcolor=black
}

\usepackage{pgf,tikz}
\usetikzlibrary{calc, arrows.meta, arrows, decorations.pathreplacing}
\usepackage{diagbox}
\usepackage{algorithm}
\usepackage{algpseudocode}

\newtheorem{Definition}{Definition}

\newtheorem{Lemma}{Lemma}
\newtheorem{Remark}{Remark}

\newcommand{\dt}{\,\partial_t\, }
\newcommand{\dx}{\,\partial_x\, }
\newcommand{\dv}{\,\partial_v\, }

\newcommand{\Dt}{\Delta t}
\newcommand{\Dx}{\Delta x}
\newcommand{\Dv}{{\Delta v}}

\newcommand{\dd}{\,\mathrm{d}\,}
\newcommand{\ddt}{\frac{\dd}{\dd t}}

\newcommand{\dxx}{\,\partial_{xx}\, }
\newcommand{\dxxx}{\,\partial_{xxx}\, }

\newcommand{\ep}{\varepsilon}

\newcommand{\N}{\mathbb{N}}

\newcommand{\R}{\mathbb{R}}
\newcommand{\TT}{\mathbb{T}}

\newcommand{\T}{\mathsf{T}}

\newcommand{\M}{\mathcal{M}}
\newcommand{\II}{\mathcal{I}}
\newcommand{\J}{\mathcal{J}}

\newcommand{\iph}{{i+\frac{1}{2}}}
\newcommand{\imh}{{i-\frac{1}{2}}}
\newcommand{\jph}{{j+\frac{1}{2}}}
\newcommand{\jmh}{{j-\frac{1}{2}}}

\newcommand{\lla}{\left\langle}
\newcommand{\rra}{\right\rangle}
\newcommand{\lbr}{\left\lbrace}
\newcommand{\rbr}{\right\rbrace}

\newcommand{\overbar}[1]{\mkern 1.5mu\overline{\mkern-1.5mu#1\mkern-1.5mu}\mkern 1.5mu}
\newcommand{\Oeps}[1]{\mathcal{O}(\ep^{#1})}

\begin{document}

\title[A space-time hybrid parareal method for kinetic equations]{A space-time hybrid parareal method for kinetic equations in the diffusive scaling}
\author{Tino Laidin}
\address[Tino Laidin]{Univ Brest, CNRS UMR 6205, Laboratoire de Mathématiques de Bretagne Atlantique, F-29200 Brest, France}
\email{tino.laidin@univ-brest.fr}

\begin{abstract}
  We present a novel multiscale numerical approach that combines parallel-in-time computation with hybrid domain adaptation for linear collisional kinetic equations in the diffusive regime. The method addresses the computational challenges of kinetic simulations by integrating two complementary strategies: a parareal temporal parallelization method and a dynamic spatial domain adaptation based on perturbative analysis.
  The multiscale parallel-in-time approach introduced in \cite{LaidinRey2025} is here considered in the diffusive scaling. It employs a coarse fluid solver for efficient temporal propagation coupled with a fine, spatially-hybridized, kinetic solver for accurate resolution. Domain adaptation is governed by two criteria: one measuring the deviation from local velocity equilibrium, and another based on macroscopic quantities available throughout the computational domain. An asymptotic preserving micro-macro decomposition framework handles the stiffness of the original problem.
  This fully hybrid methodology significantly reduces computational costs compared to full kinetic approaches by exploiting the lower dimensionality of asymptotic fluid models while maintaining accuracy through selective kinetic resolution. The method demonstrates substantial speedup capabilities and efficiency gains across various kinetic regimes.
  \textsc{2020 Mathematics Subject Classification:}
    \emph{Primary:}
      35Q20, 
      65M22, 
      82C40, 
    \emph{Secondary:}
      65M55, 
      65Y05. 

\end{abstract}

  \keywords{Kinetic equations, micro--macro decomposition, parareal, domain adaptation, multiscale methods,
  }

\maketitle

\section{Introduction}
\label{sec:Intro}

The numerical simulation of kinetic equations is a longstanding challenge due to their high dimensionality and the presence of multiple scales. A prototypical example is the Vlasov-Poisson system with linear BGK collisions, which arises in collisional plasma physics and semiconductor modeling. The unknown is the distribution function $f(t,x,v)$ depending on time $t \geq 0$, position $x \in \mathbb{T}^{d_x}$ the $d_x$-dimensional torus, and velocity $v \in \mathbb{R}^{d_v}$. The torus setting implies periodic boundary conditions in the spatial variable $x$, while the velocity variable $v$ is unbounded. The function $f$ describes the probability of observing a particle at time $t$ and position $x$ moving at velocity $v$. In the \emph{diffusive scaling}, the model reads
\begin{equation}\label{eq:kinetic}\tag{$\mathcal{P}^\ep$}
  \lbr\begin{aligned}
    &\ep \dt f^\ep + v \cdot \nabla_x f^\ep + E \cdot \nabla_v f^\ep = \frac{1}{\ep}\mathcal{Q}(f^\ep),\\
    &f(0,x,v)=f_\text{in}(x,v),
  \end{aligned}\right.
\end{equation}
where $\ep > 0$ is the Knudsen number, $E^\ep = -\nabla_x \phi^\ep$ is the self-consistent electric field, and $\mathcal{Q}(f)=\rho\M-f$ is the linear BGK operator driving relaxation toward a local equilibrium. Here $\rho=\lla f\rra=\int_{\R^{d_v}}f\dd v$ denotes the macroscopic density, and the Maxwellian is defined by
\begin{equation}
  \M(v) = \frac{1}{(2\pi)^{d_v/2}}\exp\left(\frac{|v|^2}{2}\right).
\end{equation}
The Poisson equation closes the system:
\begin{equation}\label{Poisson}
  \Delta \phi^\ep=\rho^\ep-\overbar{\rho}\quad\text{with}\quad\overbar{\rho}=\iint f_\text{in} \dd x\dd v.
\end{equation}
The collision operator conserves mass, i.e.
\begin{equation}
\int_{\mathbb{R}^{d_v}} \mathcal{Q}(f^\ep) \dd v = 0,
\end{equation}
so that the total number of particles is preserved at the continuous level. This fundamental property will also hold at the discrete and is investigated in Section~\ref{sec:Parareal}.

The Knudsen number, defined as the ratio between the mean free path of particles and a characteristic length, quantifies the importance of collisions and determines the current regime. When $\ep\sim1$, transport dominates and only few collisions occur: this is the kinetic regime. Conversely, when $\ep \ll 1$, collisions dominate and the system rapidly approaches a local Maxwellian equilibrium. In this fluid regime, macroscopic quantities drive the dynamics. In the diffusive scaling \cite{GoudonPoupaudDegond2000}, the macroscopic density $\rho^\ep$ converges towards a limit $\rho$ solution to a drift-diffusion system
\begin{equation}\label{DD}\tag{$\mathcal{P}^0$}
  \lbr\begin{aligned}
  &\dt \rho(t,x) -\nabla_x\cdot J(t,x) = 0,\quad J(t,x)=\nabla_x\rho(t,x) - E(t,x)\rho(t,x),\\
  &E = -\nabla_x \phi,\quad \Delta \phi=\rho-\overbar{\rho},\\
  &\rho(0,x)=\rho_\text{in}(x).
  \end{aligned}\right.
\end{equation}

From the computational viewpoint, equation~\eqref{eq:kinetic} suffers from two intrinsic difficulties. First, the distribution function depends on both position and velocity, leading to the curse of dimensionality. In general, the phase space is six-dimensional, making grid-based methods very costly. Second, the stiff relaxation term $(1/\ep)\mathcal{Q}(f)$ in \eqref{eq:kinetic} imposes severe constraints of the form $\Dt\leq\ep\Dx$ on explicit time discretizations. These challenges motivated the development of \emph{Asymptotic-Preserving} (AP) schemes, which remove the stiffness barrier by capturing the asymptotic limit $\ep\to0$ with a time step independant of $\ep$~\cite{Klar1999,Jin1999,DimarcoPareschi2014ActaNumerica,JinReview2022}. As shown in Figure~\ref{fig:AP}, the AP property ensures that $\mathcal{S}^\ep$ converges to $\mathcal{S}^0$ as $\ep \to 0$, with $\mathcal{S}^0$ recovering the solution to the asymptotic model in the limit of vanishing discretization parameters. While AP schemes allow time steps independent of $\ep$, they remain globally kinetic and thus computationally demanding, even in fluid regimes where the dynamics should be cheap to compute.

\begin{figure}
  \centering
  \begin{tikzpicture}[xshift=.5cm]
    \node (P1) at (0,0) {$\mathcal{S^\ep}$};
    \node (P2) at (2,0) {$\mathcal{P}^\ep$};
    \node (P3) at (0,-2) {$\mathcal{S}^0$};
    \node (P4) at (2,-2) {$\mathcal{P}^0$};

    \draw[->,dashed] (P1) -- (P2) node[midway, above] {$\Delta \to 0$};
    \draw[->,dashed] (P3) -- (P4) node[midway, below] {$\Delta \to 0$};
    \draw[->] (P1) -- (P3) node[midway, left] {$\ep \to 0$};
    \draw[->] (P2) -- (P4) node[midway, right] {$\ep \to 0$};

    \node (helperleft) at ($(P1) + (-0.9,0)$) {};
    \node (helperbottom) at ($(P3) + (0,-0.2)$) {};
  \end{tikzpicture}
  \caption{Asymptotic-Preserving diagram}
  \label{fig:AP}
\end{figure}
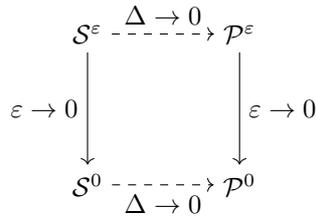

To reduce the numerical cost, various \emph{hybrid kinetic-fluid methods} have been proposed \cite{DegondDimarcoMieusens2007,KolobovArslanbekovAristovFrolovaZabelok2007,FilbetRey2015,CrestettoCrouseillesDimarcoLemou2019}. The general idea is to exploit the validity of fluid approximations near local equilibrium, restricting kinetic solvers to localized non-equilibrium subdomains. This reduces complexity by leveraging lower-dimensional macroscopic solvers in large regions of the domain. A key challenge is to design reliable criteria to decompose the domain into kinetic and fluid parts, using either macroscopic indicators or measures of the distance from equilibrium \cite{Tiwari1998,TiwariKlar1998,DegondDimarcoMieusens2007,KolobovArslanbekovAristovFrolovaZabelok2007}. In this work, we adopt the approach introduced in \cite{LevermoreMorokoffNadiga1998,Tiwari2000ApplicationOM,FilbetRey2015,FilbetXiong2018} under the hydrodynamic scaling, and later extended to the diffusive scaling in \cite{Laidin2023,LaidinRey2023}. The main idea is to rely on two criteria derived from a perturbative analysis. First, a hierarchy of macroscopic models is constructed through the Chapman--Enskog expansion of the distribution function, leading to a moment realizability criterion that quantifies the deviation of kinetic moment dynamics from their fluid counterparts. Second, the magnitude of the perturbation is used to measure the deviation from the Maxwellian equilibrium. These dynamic domain adaptation methods are typically applied cell-wise in time, although their efficiency may be reduced when kinetic regions remain spatially extended.

Other strategies have been explored, such as buffer zones with smooth transition functions \cite{DegondJinMieusens2005,DegondDimarcoMieusens2007,DimarcoMieussensRispoli2014} within the continuous model, or velocity-space hybridization based on a Maxwellian-perturbation decomposition \cite{DimarcoPareschi2008,CrestettoCrouseillesDimarcoLemou2019,HorstenSamaeyBaelmans2019,HorstenSamaeyBaelmans2020}. While these methods provide substantial savings, their benefits also remain limited in highly kinetic regimes.

In parallel, \emph{parallel-in-time algorithms} such as the parareal~\cite{LionsMadayTurinici2001,MadayTuricini2002} and the MultiGrid Reduction In Time (MGRIT)~\cite{FalgoutFriedhoffKolevMaclachlanSchroder2014,DobrevKolevPeterssonSchroder2017} methods have emerged as powerful tools to accelerate time integration of dynamical systems. The parareal method combines a coarse, inexpensive propagator with a fine, accurate solver, iteratively correcting the coarse prediction in parallel, while MGRIT extends this idea to a multilevel framework. These methods are known to converge rapidly for diffusion-dominated problems but also to face significant difficulties in advection-dominated dynamics~\cite{FarhatChandesris2003,Bal2005,Gander2008,EghbalGerberAubanel2017,NielsenBrunnerHesthaven2018,AngelGotschelRuprecht2021,HessenthalerSouthworthNordslettenRorhleFalgoutSchroder2020,SugiyamaSchroderSouthworthFriedhoff2023,DeSterckFalgoutKrzysik2023}. The kinetic regime therefore appears particularly challenging. Nevertheless, given the high cost of kinetic simulations, even slowly converging parallel-in-time algorithms can provide substantial computational savings. Conversely, since the asymptotic limit \eqref{DD} of \eqref{eq:kinetic} is diffusive, good convergence can be expected in this regime. The application of the parareal method to multiscale problems has received particular attention~\cite{DuarteMassotDescombe2011,LegollLelievreSamaey2013,GrigoriHirstoagaSalomon2023,SamaeySlawig2023,BossuytVandewalleSamaey2023}, where coarse propagators are constructed from reduced asymptotic models. In addition, recent works \cite{GanderOhlbergerRave2024,LaidinRey2025} emphasize that coarse propagation within the parareal method need not be restricted to temporal coarsening, but can instead involve coarser spatial grids or dimension reduction, a perspective particularly relevant for kinetic equations. These strategies naturally connects with the Heterogeneous Multiscale Method (HMM) framework~\cite{WeinanEngquist2003,Weinan2011PrinciplesOM,AbdulleWeinanEngquistVanden-Eijnden2012}.

Building on these ideas, a kinetic-fluid parareal method was recently introduced in \cite{LaidinRey2025}. In this setting, the fluid, hyperbolic model provides the coarse solver, while the kinetic solver captures fine-scale effects. In this scaling, efficiency is delicate. The combination of hyperbolic (presence of shocks) and kinetic (transport-dominated) solvers can hinder convergence. In our setting, which corresponds to the case $\alpha=1$ in the framework of \cite{LaidinRey2025}, applying a diffusive scaling eliminates the challenges of the hyperbolic limit, replacing them with a dynamics more parareal-compatible. We note that the diffusive scaling was not implemented in \cite{LaidinRey2025}, and consequently the behavior of the parareal multiscale method in this regime was not investigated there.

The present work introduces a novel space-time multiscale numerical methodology that combines three ingredients namely AP schemes, hybridization in space, and parallel-in-time acceleration into a unified framework for the Vlasov--Poisson--BGK system in the diffusive scaling. The key components are:
\begin{itemize}
\item an \emph{asymptotic-preserving micro-macro scheme} that handles stiffness and ensures consistency with the diffusion limit;
\item a \emph{hybrid kinetic-fluid phase-space solver}, dynamically adapting the computational domain based on deviation from local equilibrium and macroscopic indicators;
\item a \emph{multiscale parallel-in-time strategy}, coupling the fluid solver as a coarse propagator with the above hybrid solver as the fine propagator via the parareal algorithm.
\end{itemize}
This combination leverages the robustness of AP schemes in stiff regimes, the dimensional reduction of hybrid solvers, and the scalability of the parareal algorithm. The resulting method yields significant computational savings compared to an AP, but fully kinetic approach, while preserving accuracy across both kinetic and fluid regimes.

The paper is organized as follows. Section~\ref{sec:Hybrid} recalls the dynamic domain adaptation and numerical schemes. Section~\ref{sec:Parareal} presents the parallel-in-time algorithm and its coupling with the hybrid solver, leading to a space-time hybrid solver. The implementation of the method is briefly discussed in Section~\ref{sec:Implementation} and section~\ref{sec:Num} presents numerical experiments illustrating the efficiency and accuracy of the method.

\section{Dynamic domain adaptation}
\label{sec:Hybrid}
In this section, we recall the building blocks of the dynamic domain adaptation introduced in \cite{Laidin2023,LaidinRey2023} for kinetic equations in the diffusive scaling. We start by presenting a hierarchy of macroscopic models used to define coupling criteria and then present the numerical schemes used.

\subsection{Macroscopic hierarchy}
To obtain a macroscopic coupling criterion, a hierarchy of macroscopic models is built from an asymptotic expansion of the distribution. We introduce the truncated Chapman-Enskog expansion of $f^\ep$ at order $L\in\N$:
\begin{equation}\label{eq:ChapmanEnskog}
    f^\ep(t,x,v)=\rho^\ep(t,x)\M(v) + \sum^L_{l=1}\ep^l h^{(l)}(t,x,v).
\end{equation}
By inserting \eqref{eq:ChapmanEnskog} into the original equation \eqref{eq:kinetic}, one can identify powers of epsilon to obtain 
\begin{subequations}\label{Identification}
\begin{alignat}{2}
l=0:&\quad &h^{(1)} = &-\T(\rho^\ep\M),\\
l=1:&\quad &h^{(2)} = &-\dt(\rho^\ep\M) -\T(h^{(1)}),\\
2\leq l\leq L-1:&\quad &h^{(l+1)} = &-\dt h^{(l-1)} - \T(h^{(l)}),
\end{alignat}
\end{subequations}
where we use the notation $\T f=v \cdot \nabla_x f + E \cdot \nabla_v f$ for the transport operator appearing in \eqref{eq:kinetic}. In particular, at zeroth order ($L=0$), substituting $f^\varepsilon = \rho^\varepsilon\M$ into \eqref{eq:kinetic} and using that $\mathcal{Q}(\rho^\varepsilon\M)=0$ yields $\varepsilon \partial_t(\rho^\varepsilon\M) + \T(\rho^\varepsilon\M) = 0$.
The hierarchy of macroscopic models is obtained by further truncations of \eqref{eq:ChapmanEnskog} plugged into \eqref{eq:kinetic} and integrated in velocity. Note that the first order ($L=1$) allows us to formally recover the asymptotic model \eqref{DD}.

Let us now recall the macroscopic model obtained by setting $L=3$. In the $d_x=1$ and $d_v=3$ setting, it was shown in \cite[Section 2]{LaidinRey2023} that the higher order macroscopic model writes
\begin{equation}\label{HigherOrderModel}
    \dt\rho^\ep-\dx J^\ep=-\ep^2\mathcal{R^\ep}+\Oeps{4},
\end{equation}
where the remainder $\mathcal{R^\ep}$ is given by
\begin{equation}\label{Remainder}
    \mathcal{R}^\ep=-\dxxx J^\ep + E\dxx J^\ep + (2\rho^\ep-3\overbar{\rho})\dx J^\ep +2J^\ep\dx\rho^\ep -\dx\rho^\ep\dt E^\ep.
\end{equation}
In the following, the $d_x=1$ and $d_v=3$ will be considered and extensions of \eqref{HigherOrderModel} to higher dimensions are discussed in \cite[Appendix A]{LaidinThese}. Additionnaly, the quantity $\mathcal{R}^\ep$ depending only on macroscopic quantities, we also define $\mathcal{R}$, the remainder \eqref{Remainder}, computed with the macroscopic quantities solution to \eqref{DD}:
\begin{equation}\label{RemainderBis}
    \mathcal{R}=-\dxxx J + E\dxx J + (2\rho-3\overbar{\rho})\dx J +2J\dx\rho -\dx\rho\dt E.
\end{equation}

\subsection{Domain indicators}
\label{subsec:Indicators}

The idea of the dynamic domain decomposition is to approximate the kinetic system by its fluid counterpart whenever possible, thereby exploiting the reduced computational cost of fluid solvers. To ensure accuracy, the domain partition, in kinetic and fluid subdomains, must dynamically track the local state of the solution. At each time step, the status of each spatial cell is updated according to two coupling thresholds, $\delta_0$ and $\eta_0$, defined through the following indicators:

\begin{itemize}
\item If $\left|\mathcal{R}^{\ep}\right| > \delta_0$ in a kinetic cell or $\left|\mathcal{R}\right| > \delta_0$ in fluid cells, then the dynamics of the solution to \eqref{HigherOrderModel} deviates too much from that of the limit model \eqref{DD}, and the fluid description cannot be employed;
\item Let $g^\ep = f^\ep - \rho^\ep \M$. If $\left\| g^\ep \right\| > \eta_0$, the solution is far from local equilibrium in velocity space, and the fluid description is not admissible.
\end{itemize}

If at least one of these conditions is satisfied, the corresponding cell must remain (or switch to) kinetic. In practice, the thresholds $\delta_0$ and $\eta_0$ are chosen to balance accuracy and efficiency: too large values may trigger the fluid approximation prematurely, while too small values may reduce computational gains. Furthermore, as $\mathcal{R}^\ep$ and $\mathcal{R}$ involve high-order derivatives, the finite differences presented in Table~\ref{tab:finiteDiff} will be used for their approximation in each cell. Regarding the computation of the criterion on the perturbation, a classical $L^2$-norm is used. Finally, in the fluid domain, since no kinetic data is available, one does not have access to the velocity dependent perturbation, and only the macroscopic criterion $|\mathcal{R}|$ is used.

\begin{table}
    \centering
    \begin{tabular}{|c|c|c|c|c|c|c|c|}
        \hline
        \diagbox{Derivative}{Index} & -3 & -2 & -1 & 0 & 1 & 2 & 3\\
        \hline
        1 & -1/60 & 3/20 & -3/4 & 0 & 3/4 & -3/20 & 1/60\\
        \hline
        2 & 1/90 & -3/20 & 3/2 & -49/18 & 3/2 & -3/20 & 1/90\\
        \hline
        3 & 1/8 & -1 & 13/8 & 0 & -13/8 & 1 & -1/8\\
        \hline
        4 & -1/6 & 2 & -13/2 & 28/3 & -13/2 & 2 & -1/6\\
        \hline
    \end{tabular}
    \caption{Central finite differences coefficients.}
    \label{tab:finiteDiff}
\end{table}

\subsection{Numerical schemes}

\textbf{Mesh and discretization.}  
We restrict ourselves to the case $d_x=1$ and $d_v=3$ used in the numerical experiments. Let $\II=\{1,\dots,N_x\}$ and $\mathcal{J}=\{1,\dots,N_v\}$. The $(x,v)$--phase space is discretized in a finite-volume fashion with control volumes $K_{ij}=\mathcal{X}_i\otimes\mathcal{V}_j$, for $i\in\II$ and $j=(j_x,j_y,j_z)\in\mathcal{J}^3$. We denote by $\Delta x$ and $\Delta v^3$ the uniform volumes of $\mathcal{X}_i$ and $\mathcal{V}_j$ respectively.  

\noindent\textbf{Micro--macro scheme.}  
Following \cite{Laidin2023,LaidinRey2023}, we use a finite volume relaxed AP micro--macro scheme introduced in \cite{LemouMieussens2008,Lemou2010,CrouseillesLemou2011}, based on the decomposition
\begin{equation}
  f(t,x,v) = \rho(t,x)\M(v) + g(t,x,v),
\end{equation}
where $g$ is the perturbation from the Maxwellian. The resulting scheme writes
\begin{equation}\label{eq:SMM}\tag{$\mathcal{S}^\text{MM}$}
  \lbr\begin{aligned}
    g_{\iph,j}^{\ep,n+1} &= g_{\iph,j}^{\ep,n}e^{-\Dt/ \ep^2} - \ep(1-e^{-\Dt/ \ep^2})
        \left(\frac{T^{\ep,n}_{\iph,j}}{\Dx\Dv^3}+\xi_j\M_jJ^{\ep,n}_\iph\right),\\
    \rho_i^{\ep,n+1} &= \rho_i^{\ep,n} - \frac{\Dt}{\ep\Dx}
        \left(\langle \xi g^{\ep,n+1}_\iph \rangle_\Delta-\langle \xi g^{\ep,n+1}_\imh \rangle_\Delta \right), \\
    E^{\ep,n}_\iph&-E^{\ep,n}_\imh = (\rho_i^{\ep,n}-\overbar{\rho})\Dx,
  \end{aligned}\right.
\end{equation}
where $\xi_j$ denotes the $x$-component of $v_j$, $\langle \cdot\rangle_\Delta=\sum_{\mathcal{J}}\cdot_j\Dv$ is a first order quadrature rule in velocity, and $T^{\ep,n}_{\iph,j}$ is an upwind discretization of the transport term $\T g - \dx\lla v_x g\rra\M$. More precisely, adopting a finite volume framework and denoting $r^\pm=(|r|\pm r)/2$, the operators $v\dx$ and $E\dv$ are approximated by
\begin{equation*}
  v\dx f(x_i,v_j)\approx\frac{1}{\Dx}\left(\mathcal{F}_{\iph,j}-\mathcal{F}_{\imh,j}\right),\text{ with}\quad \mathcal{F}_{\iph,j}= \Dv\left(v_j^+ \,f^n_{ij} + v_j^- \,f^n_{i+1,j}\right)
\end{equation*}
and
\begin{equation*}
  E\dv f(x_i,v_j)\approx\frac{1}{\Dv}\left(\mathcal{G}_{i,\jph}-\mathcal{G}_{i,\jmh}\right),\text{ with}\quad \mathcal{G}_{i,\jph}= \Dx\left(E_i^+ \,f^n_{ij} + E_i^- \,f^n_{i,j+1}\right)
\end{equation*}

For fixed $\Dx,\Dv>0$, the scheme is AP in the diffusion limit, independently of the initial data \cite{Laidin2023}.  

\noindent\textbf{Limit scheme.}  
As $\ep\to0$, one recovers a discretization of \eqref{DD}:
\begin{equation}\label{eq:LimitScheme}\tag{$\mathcal{S}^0$}
  \lbr\begin{aligned}
    \rho_i^{n+1} &= \rho_i^{n} +  m_2^{\Dv}\frac{\Dt}{\Dx}\left(J_\iph^{n}-J_\imh^{n}\right), \\
    J_\iph^{n} &= \frac{\rho^n_{i+1}-\rho^n_{i}}{\Dx} - E_\iph\rho^n_\iph,\quad
    E^{n}_\iph- E^{n}_\imh = (\rho_i^{n}-\overbar{\rho})\Dx,
  \end{aligned}\right.
\end{equation}
with $m_2^{\Dv} =\sum_{l\in\mathcal{J}}v_l^2M_l\Dv_l$. The derivation of this limit scheme is outside the scope of this paper but it follows from plugging the scheme on $g^{\ep}$ into the update of $\rho^\ep$ and taking the limit $\ep\to0$ (see \textit{e.g.} \cite{CrouseillesLemou2011,Laidin2023}).

\noindent\textbf{Hybrid scheme.}  
Combining \eqref{eq:SMM} and \eqref{eq:LimitScheme} then yields the hybrid micro--macro scheme. The cells of the computational domain are partitionned in a fluid subdomain $\Omega_\mathcal{F}^n$ and a kinetic subdomain $\Omega_\mathcal{K}^n$ at time $n$. The AP scheme~\eqref{eq:SMM} is applied in kinetic cells within $\Omega_\mathcal{K}$, while the limit scheme~\eqref{eq:LimitScheme} is used in cells of $\Omega_\mathcal{F}$.
\begin{equation}\label{eq:SHMM}\tag{$\mathcal{S}^\text{HMM}$}
\left\{
\begin{aligned}
g_{\iph,j}^{\ep,n+1} &= g_{\iph,j}^{\ep,n} e^{-\Dt/\ep^2}
  - \ep\!\left(1-e^{-\Dt/\ep^2}\right)
    \left( \frac{T^{\ep,n}_{\iph,j}}{\Dx\Dv^3}
           + \xi_j M_j J^{\ep,n}_\iph \right),
&\,&
\mathcal{X}_{\iph}\in\Omega^n_\mathcal{K},
\\[0.45em]
g_{\iph,j}^{\ep,n+1} &= 0,
&\,&
\mathcal{X}_{\iph}\in\Omega^n_\mathcal{F},
\\[0.75em]
\rho_i^{\ep,n+1} &= \rho_i^{\ep,n}
  + \frac{\Dt}{\Dx}\,\widetilde{J}_i^{\ep,n},
&\,&
\\[0.75em]
\widetilde{J}_i^{\ep,n} &=
  \begin{cases}
  -\dfrac{1}{\ep}\big(
     \langle \xi g^{\ep,n+1}_\iph\rangle_\Delta
     - \langle \xi g^{\ep,n+1}_\imh\rangle_\Delta
     \big),\\[0.65em]
  m_2^{\Delta v}\,(J_\iph^{n}-J_\imh^{n}),
  \end{cases}
&\,&
\begin{aligned}
\mathcal{X}_i\in\Omega^n_\mathcal{K},\\[0.65em]
\mathcal{X}_i\in\Omega^n_\mathcal{F},
\end{aligned}
\\[0.75em]
E^{\ep,n}_\iph &- E^{\ep,n}_\imh = (\rho_i^{\ep,n}-\overbar{\rho})\,\Dx.
&\,&
\end{aligned}
\right.
\end{equation}
with the update of $\Omega^n_\mathcal{K}$ and $\Omega^n_\mathcal{F}$ detailed in Section~\ref{subsec:Indicators}. We denote by $\mathcal{R}^{\ep,n}_i$ (resp. $\mathcal{R}^{n}_i$) the approximation of $\mathcal{R}^{\ep}(t^n,x_i)$ (resp. $\mathcal{R}(t^n,x_i)$) using Table~\ref{tab:finiteDiff}. Considering a spatial cell $\mathcal{X}_i$, one distintinguishes two cases:
\begin{itemize}
  \item  If $\mathcal{X}_i\subset\Omega^n_\mathcal{F}$:
        \begin{itemize}
            \item If $\left|\mathcal{R}^{n}_i\right|\leq\eta_0$, then the cell stays fluid at $t^{n+1}$: $\mathcal{X}_i\subset\Omega^{n+1}_\mathcal{F}$.
            \item If $\left|\mathcal{R}^{n}_i\right|>\eta_0$, then the cell becomes kinetic at $t^{n+1}$: $\mathcal{X}_i\not\subset\Omega^{n+1}_\mathcal{F}$ and $\mathcal{X}_i\subset\Omega^{n+1}_\mathcal{K}$.
        \end{itemize}
  \item  If $\mathcal{X}_i\subset\Omega^n_\mathcal{K}$:
      \begin{itemize}
        \item If $\,\|g^{\ep,n}_{\imh}\|_{2}>\delta_0$ and $\|g^{\ep,n}_{\iph}\|_{2}>\delta_0$ then the cell stays kinetic at $t^{n+1}$: $\mathcal{X}_i\subset\Omega^{n+1}_\mathcal{K}$
        \item If $\,\|g^{\ep,n}_{\imh}\|_{2}\leq\delta_0$, $\,\|g^{\ep,n}_{\iph}\|_{2}\leq\delta_0$ and $\left|\mathcal{R}^{\ep,n}_i\right|>\eta_0$, then the cell stays kinetic at $t^{n+1}$: $\mathcal{X}_i\subset\Omega^{n+1}_\mathcal{K}$
        \item If $\,\|g^{\ep,n}_{\imh}\|_{2}\leq\delta_0$, $\,\|g^{\ep,n}_{\iph}\|_{2}\leq\delta_0$ and $\left|\mathcal{R}^{\ep,n}_i\right|\leq\eta_0$, then the cell becomes fluid at $t^{n+1}$: $\mathcal{X}_i\not\subset\Omega^{n+1}_\mathcal{K}$ and $\mathcal{X}_i\subset\Omega^{n+1}_\mathcal{F}$.
      \end{itemize}
\end{itemize}

Note that the perturbation $g$ is set to $0$ in $\Omega_\mathcal{F}$, corresponding to the assumption that the solution is exactly at equilibrium. A finer treatment would replace $g$ by its expansion in~\eqref{eq:ChapmanEnskog}, which can be computationally expensive when evaluated repeatedly, or rely on a multilevel hybridization as proposed in~\cite{CaparelloPareschiRey2025}. In practice, for efficiency, the array representing $g$ is not updated in $\Omega_\mathcal{F}$ and is modified only when a state change occurs.

\section{Multiscale parareal algorithm}
\label{sec:Parareal}

In the spirit of \cite{LaidinRey2025}, we now construct a multiscale parareal algorithm based on the two solvers \eqref{eq:SHMM} and \eqref{eq:LimitScheme}.

\subsection{Parallel in time integration}
Let us now recall the parareal algorithm adapted to our multiscale PDE setting. We are interested in the accurate approximation of the macroscopic density $\rho$ and not on the full kinetic distribution itself. Assuming continuous time, and denoting by $\left(u(t)\right)_{i\in\II}$ a generic discretized density , it is solution to an ODE of the form
\begin{equation*}
    \lbr\begin{aligned}
    &\ddt u(t) = \mathcal{S}(u),\quad t\in[0,T],\\
    &u(0) = u_\text{in},
    \end{aligned}\right.
\end{equation*}
where $\mathcal{S}$ stands for the numerical methods used to update the density. The goal of the parareal algorithm is to approximate the solution at some fixed discrete times $T^n$, which are not required to be uniformly spaced,
\begin{equation*}
    u^n \approx u(T^n),\quad n\in\{0,\dots,N_g\},\, N_g\in\N.
\end{equation*}
To do so, instead of using a single time integrator (or propagator), two solvers are coupled through an iterative process: the \emph{parareal iterations}. The first solver is denoted by $\mathcal{G}(T^n,T^{n+1},u^n)$. It must be cheap to compute numerically and provides some numerical \emph{guesses} of the behavior of the solution. The second solver, that we denote by $\mathcal{F}(T^n,T^{n+1},u^n)$, must be very accurate, and is used to improve the aforementioned guesses. In both cases, the solvers stand for generic propagators to update $u^n$ to $u^{n+1}$. Their definitions will be specified in Section~\ref{subsec:PararealMulti}.

While the propagator $\mathcal{G}$ is less costly, its drawback is naturally its precision. However, it is corrected by the accurate, computationally expensive, propagator $\mathcal{F}$. The algorithm then unfolds as follows:
\smallskip
\begin{enumerate}
    \item Divide the time domain in $N_g\in\N$ subintervals $[T^n,T^{n+1}]$, $n\in\{0,\dots,N_g-1\}$;
    \item Perform a first coarse guess: 
        \begin{equation*}
            u^{n+1,0} = \mathcal{G}(T^n,T^{n+1},u^{n,0}) = u_\text{in};
        \end{equation*}
    \item Refine the guess through the parareal iterations:
    \begin{equation*}
        u^{n+1,k+1} = \mathcal{G}(T^n,T^{n+1},u^{n,k+1}) + \mathcal{F}(T^n,T^{n+1},u^{n,k}) - \mathcal{G}(T^n,T^{n+1},u^{n,k}),
    \end{equation*}
    where $k=1,2,\dots$ denotes the $k$\textsuperscript{th} parareal iteration.
\end{enumerate}
\smallskip
The advantage of this iterative algorithm is its ability to compute the expensive fine propagations in parallel. Therefore, as long as the number of parareal iterations remains small enough to obtain a given accuracy, one can expect a reduction of the computational cost of the time integration.

\subsection{Linking the scales}
A crucial point of multiscale algorithms is about linking the various scales. Thanks to the micro-macro framework, the unknown of the system are now the perturbation at the kinetic level and the density at the fluid one. Consequently, macroscopic data are available at any point in space and any time, without the need of projecting a distribution towards its moments. Nonetheless, one still needs to define a lifting procedure to recover the perturbation. To lift a macroscopic density $\rho(t,x)$ to a phase-space perturbation, we build upon the Chapman-Enskog expansion \eqref{eq:ChapmanEnskog}. The perturbation, that is the unknown of scheme \eqref{eq:SHMM}, can be reconstructed as follows:
\begin{Definition}\label{def:lift}
    Let $L\in\N$. For a macroscopic quantity $\rho(t,x)$, the lifting of order~$L$, denoted $\mathcal{L}^L\rho(t,x,v) = g(t,x,v)$, is defined as
    \begin{equation*}
         \mathcal{L}\rho = g =\sum_{l=1}^{L}\ep^l h^{(l)},
    \end{equation*}
    where the perturbations $h^{(l)}$ depend on velocity moments of $\M$ and on the derivatives of $\rho$ and $E$.
\end{Definition}
It is worth mentioning that the mass of the perturbation $g$ is $0$ by construction. More precisely, the zeroth order moment of $g$ in velocity is $0$. This property is of particular importance in the discrete setting to ensure mass conservation in the end. Indeed, the mass is only be supported by the density $\rho$ in the micro-macro decomposition $f=\rho\M+g$.

\subsection{Multiscale algorithm}\label{subsec:PararealMulti}
Once the parallel in time integration is defined, it can be adapted to the multiscale setting thanks to a discrete counterpart of the lifting operator $\mathcal{L}$. Taking advantage of the micro-macro framework, the macroscopic quantity $\rho$ is always available from \eqref{eq:SHMM} and \eqref{eq:SMM}. It only remains to reconstruct the perturbation $g$. For all $(i,j)\in\II\times\J$, we define
\begin{equation*}
    (\mathcal{L}\rho)_{ij} = g_{ij} = \sum_{l=1}^{L}\ep^l h^{(l)}_{ij}.
\end{equation*}
With this definition, one can actually check that the discrete integral in velocity of such $g_{ij}$ is $0$ and deduce that the discrete perturbation has zero mass. To approximate the perturbations $h^{(l)}$, high-order finite differences are used (See Table~\ref{tab:finiteDiff}). It is worth mentioning that for $L=3$, the perturbation involves up to fourth order derivatives of the density and third order on the electrical field. In addition, the lifting is defined at cell centers but the scheme \eqref{eq:SMM} requires a perturbation at interfaces. In practice, average interface values of the density can be used in the lifting procedure.

Let us now precise the fine and coarse solvers. We set $\mathcal{F} = \mathcal{S}^\text{HMM}$ the hybrid micro-macro scheme that takes as input a macroscopic density and a perturbation. The coarse solver is set to $\mathcal{G} = \mathcal{S}^0$ and only depends on the density. The fully multiscale method is summarized in Algorithm~\ref{algo:paraMS}.

Another crucial point of the method is the computation of the electrical field from the coupling with the Poisson equation. We consider explicit in time discretization and the electrical field is therefore obtained from the density when needed. In particular, the parareal correction is not applied to the electrical field. The corrected field is computed from the corrected density.

\begin{algorithm}
    \caption{Multiscale kinetic parareal Algorithm}\label{algo:paraMS}
    \begin{algorithmic}[1]
    \Require $\rho^{0,0}$ and $g^{0,0}$
    \State Solve for $E^{0,0}$ using $\rho^{0,0}$
    \For{$n = 1,\dots, N_g$} \Comment{First coarse guess}
        \State $\rho^{n,0}\gets \mathcal{S}^0\left(\rho^{n-1,0}\right)$ 
        \State Solve for $E^{n,0}$ using $\rho^{n,0}$
        \State $g^{n,0}\gets \mathcal{L}^L\left(\rho^{n,0}\right)$ 
    \EndFor
    \vspace*{.2cm}
    \While{$k \leq K$ \textbf{or} error $\geq$  tol}\Comment{Parareal iterations}
    \vspace*{.2cm}
        \For{$n = k,\dots, N_g$} \Comment{Compute density jumps in parallel}
            \State $\Delta^{n,k} \gets \mathcal{S}^\text{HMM}(\rho^{n-1,k}, g^{n-1,k}) - \mathcal{S}^0\left(\rho^{n-1,k}\right)$
        \EndFor
        \vspace*{.2cm}
        \For{$n = k,\dots, N_g$} \Comment{Sequential correction}
            \State $\rho^{n,k+1} \gets \mathcal{S}^0\left(\rho^{n-1,k}\right) + \Delta^{n,k}$
            \State Solve for $E^{n,k}$ using $\rho^{n,k}$
            \State $g^{n,k+1}\gets \mathcal{L}^L\left(\rho^{n,k+1}\right)$ 
        \EndFor
        \vspace*{.2cm}
        \State{Compute successive error $\max_n\lbr\left|\rho^{n,k}-\rho^{n,k-1}\right|\rbr$ and $k\gets k+1$}
    \EndWhile
    \end{algorithmic}
\end{algorithm}

\begin{Lemma}[Conservation of mass]\label{lem:massCons}
  Let $(\rho_i^{n,k})_{(n,i)\in\II\times\{0,\dots,N_g\}}$ be given by Algorithm~\ref{algo:paraMS} and assume exact mass conservation of the hybrid numerical method \eqref{eq:SHMM}. We denote the mass of the system at time $t^n$ and parareal iteration $k$ as
  \begin{equation}
    m^{n,k} = \sum_{i\in\II} \rho_i^{n,k}\Dx.
  \end{equation}
  Then, for all $k>0$ and all $n\in\{0,\dots,N_g\}$,
  \begin{equation}\label{eq:MassCons}
    m^{n,k+1} = m^{0,0}.
  \end{equation}
\end{Lemma}
\begin{proof}
  A first observation is that, by construction of the parareal algorithm, for all $k\geq0$, 
  \begin{equation}\label{eq:massInit}
    m^{0,k}=m^{0,0},
  \end{equation}
  since the initial data is fixed. Then, let $n\in\{0,\dots,N_g\}$ and let $\left(\rho_{i}^n\right)_{i\in\II}$ be given by Algorithm~\ref{algo:paraMS}. By construction of the asymptotic scheme \eqref{eq:LimitScheme} and the assumption on the hybrid method \eqref{eq:SHMM} one has the following properties:
  \begin{equation}\label{eq:schemeCons}
    \sum_{i\in\II} \mathcal{S}^0\left(\rho^{n,k}\right)_i \Dx = m^{n,k},\text{ and}\quad \sum_{i\in\II} \mathcal{S}^\text{HMM}(\rho^{n,k})_i\Dx = m^{n,k}.
  \end{equation}
  In addition, updating the perturbation with \eqref{eq:SMM} preserves its $0$ mass by construction. This property carries over under the action of the hybrid scheme \eqref{eq:SHMM} as this property only depends on the velocity variable and not on the spatial coupling. Using the update relation of the parareal algorithm, one has:
  \begin{equation}
      m^{n,k+1} = \sum_{i\in\II} \mathcal{S}^0\left(\rho^{n-1,k+1}\right)_i \Dx + \sum_{i\in\II} \mathcal{S}^{\text{HMM}}\left(\rho^{n-1,k-1}\right)_i \Dx - \sum_{i\in\II} \mathcal{S}^0\left(\rho^{n-1,k-1}\right)_i \Dx
  \end{equation}
  Thanks to \eqref{eq:schemeCons}, it follows that for all $k\geq0$
  \begin{equation}
      m^{n,k+1} = m^{n-1,k+1}
  \end{equation}
  One then obtains \eqref{eq:MassCons} by induction over $n$, using \eqref{eq:massInit}.
\end{proof}

\begin{Remark}
  As pointed out in \cite{Laidin2023}, the hybrid method \eqref{eq:SHMM} is not exactly conservative. However, the mass variation is in practice negligible for reasonable coupling thresholds. In addition, the exact mass variation decays very fast as $\ep\to0$ (See Lemma 4.2 in \cite{Laidin2023}). Exact mass conservation could be achieved using, for example, different boundary conditions between kinetic and fluid states following the approaches in \cite{FilbetRey2015,CaparelloPareschiRey2025} by matching fluxes.
\end{Remark}

\section{Implementation}
\label{sec:Implementation}

The computational gain of Algorithm~\ref{algo:paraMS} compared to the standard AP scheme \eqref{eq:SMM} depends on two main factors. First, parallelization allows the computational load of the kinetic propagations that are needed to compute the correction, to be distributed efficiently. Second, the implementation of dynamic domain adaptation accelerates computations independently of parallel resources, providing speed-ups even without multiple processors, unlike the parareal algorithm.

Algorithm~\ref{algo:paraMS} is currently implemented in Fortran08 using a shared-memory framework with OpenMP. To minimize overhead from memory allocations and deallocations, arrays required for the hybrid propagation, that are executed in parallel, are preallocated per thread. In a distributed-memory framework, these considerations are replaced by communication management. Moreover, since the workload of the parallel loop can vary significantly depending on proximity to equilibrium, a dynamic scheduling was found to offer the best performance, despite the small additional overhead.

Assuming the cost of communications is negligible compared to the solver times, which is reasonable in our shared-memory framework, we denote by $T_{\text{HMM}}$ (resp. $T_{\text{Fluid}}$) the maximum time for \eqref{eq:SHMM} (resp. \eqref{eq:LimitScheme}) to evolve a given initial data from $T^n$ to $T^{n+1}$. The maximum is taken over all propagations because, on each sub-time interval, the stability conditions and the domain adaptation may vary. In addition, the cost of the lifting operator, which is fully non-local in both position (due to derivative computations) and velocity, can be significant; we denote this cost by $T_{\text{Lift}}$. The ideal numerical cost of Algorithm~\ref{algo:paraMS} on $N_p$ threads with $N_g$ sub-time intervals is then given by the formula:

\begin{equation*}
    T_{\text{Parareal}} = T_{\text{Fluid}} + N_g k \left(\frac{T_{\text{Lift}}+T_{\text{HMM}}+T_{\text{Fluid}}}{N_p} + T_{\text{Fluid}}\right).
\end{equation*}
Consequently, the ideal number of parareal iterations to be less costly than a fully kinetic simulation should satisfy:
\begin{equation*}
    T_{\text{Parareal}}\leq N_gT_{\text{HMM}}
\end{equation*}
or stated differently,
\begin{equation*}
    k_{\text{opt}} \leq \left\lceil\frac{N_gT_{\text{HMM}}-T_{\text{Fluid}}}{N_g\left(\frac{T_{\text{Lift}}+T_{\text{HMM}}+T_{\text{Fluid}}}{N_p}+T_{\text{Fluid}}\right)}\right\rceil.
\end{equation*}
The performance of the multiscale algorithm was evaluated on the architecture presented in Table~\ref{tab:architecture}.

\begin{table}
    \centering
    \begin{tabular}{c|c}
        \# CPUs &  Intel(R) Xeon(R) Platinum 8268 CPU @ 2.90GHz x2\\
        \hline
        \# cores & 48 \\
        \hline
        RAM & 750 GB \\
    \end{tabular}
    \caption{Computing architecture.}
    \label{tab:architecture}
\end{table}

\section{Numerical results}
\label{sec:Num}

For practical purposes, the velocity domain is truncated to the box $[-v_\star,v_\star]$, and we denote by $x_\star$ the length of the 1-dimensional torus $\TT$. Unless stated otherwise, the discretization parameters are fixed as
\begin{equation*}
N_x = 32,\quad N_{v_x} = 32,\quad N_{v_y} = N_{v_z} = 16,\quad x_\star = 2\pi,\quad v_\star = 8.
\end{equation*}
We consider the following far-from-equilibrium initial condition:
\begin{equation*}
f(0,x,v) = \M(v_x,v_y,v_z), v_x^4 \Bigl(1 + 0.9 \cos\left(\tfrac{2\pi x}{x_\star}\right)\Bigr).
\end{equation*}

The thresholds for the dynamic domain adaptation indicators are set to $\eta_0 = \delta_0 = 10^{-5}$. These empirical values are an order of magnitude lower than the ones used in \cite{Laidin2023} where their impact and accuracy was investigated. These thresholds are chosen as a good compromise between accuracy and acceleration. The perturbation $g$ is reconstructed at order $3$, which provides the most accurate results compared to lower-order reconstructions and is consistent with the derivation of the macroscopic indicators.

The number of coarse time intervals is set to $N_g=200$. This value guarantees convergence across all regimes. While regimes with $\ep \ll 1$ require relatively few coarse intervals, it was observed that in the kinetic, transport-dominated regime ($\ep \sim 1$), convergence is slower, and a larger value of $N_g$ is necessary.

Figures~\ref{fig:snapshotsDensity05} and \ref{fig:snapshotsEfield05} shows snapshots of the densities and electric field computed with Algorithm~\ref{algo:paraMS} for $\ep=0.5$, with and without activation of the multiscale procedures (parareal integration and domain adaptation). The corresponding results for $\ep=10^{-4}$ are shown in Figures~\ref{fig:snapshotsDensity0001} and \ref{fig:snapshotsEfield0001}. In both the kinetic regime ($\ep=0.5$) and the fluid regime ($\ep=10^{-4}$), the different methods agree closely. Furthermore, the asymptotic-preserving property is confirmed, as the hybrid solutions all coincide with the limiting fluid solution for small $\ep$.

\begin{figure}
  \includegraphics[width=.83\linewidth]{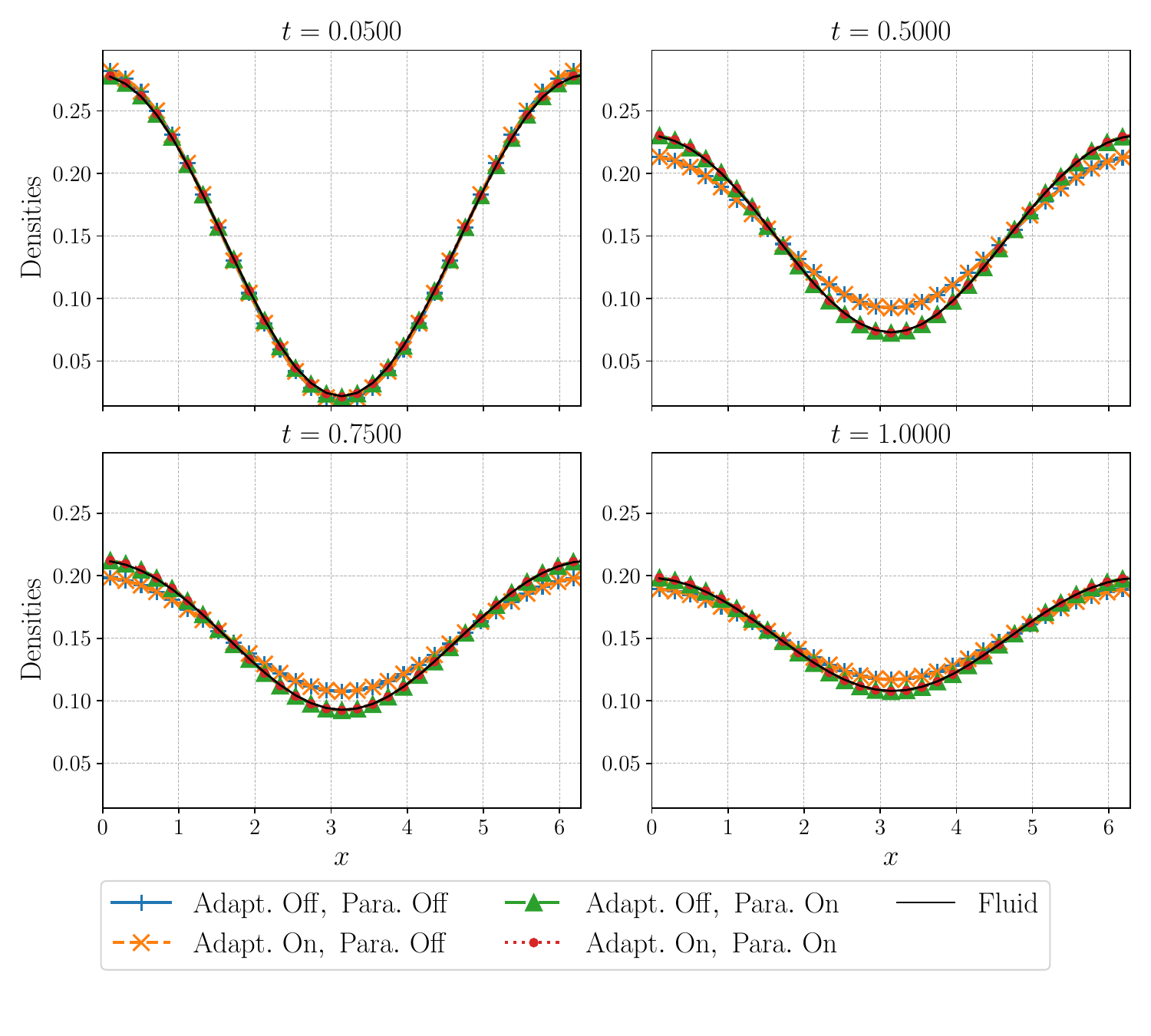}
  \caption{Density snapshots for $\ep=0.5$ obtained with Algorithm~\ref{algo:paraMS}.}
  \label{fig:snapshotsDensity05}
\end{figure}

\begin{figure}
  \includegraphics[width=.83\linewidth]{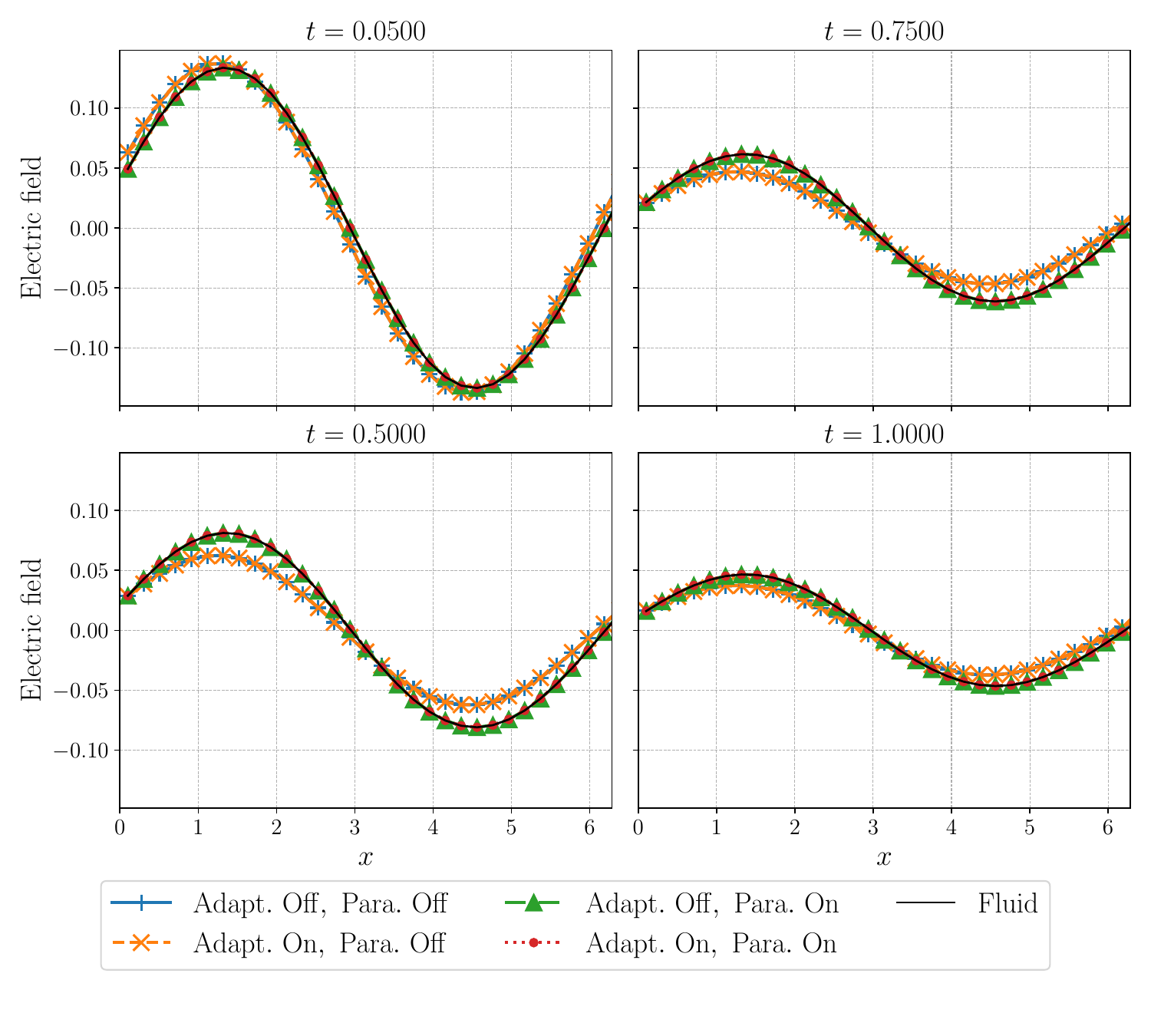}
  \caption{Electric field snapshots for $\ep=0.5$ obtained with Algorithm~\ref{algo:paraMS}.}
  \label{fig:snapshotsEfield05}
\end{figure}

\begin{figure}
  \includegraphics[width=.83\linewidth]{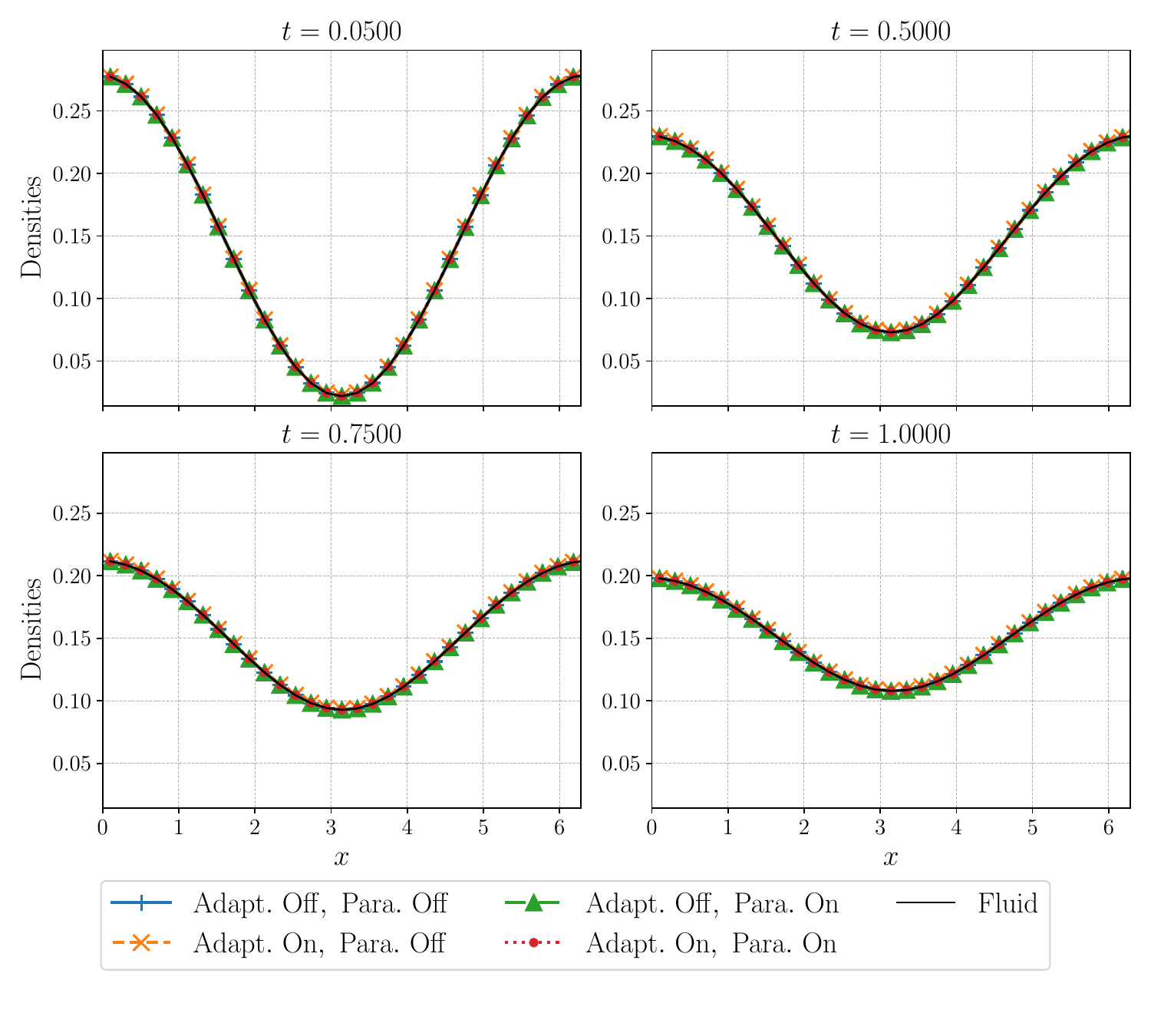}
  \caption{Density snapshots for $\ep=10^{-4}$ obtained with Algorithm~\ref{algo:paraMS}.}
  \label{fig:snapshotsDensity0001}
\end{figure}

\begin{figure}
  \includegraphics[width=.83\linewidth]{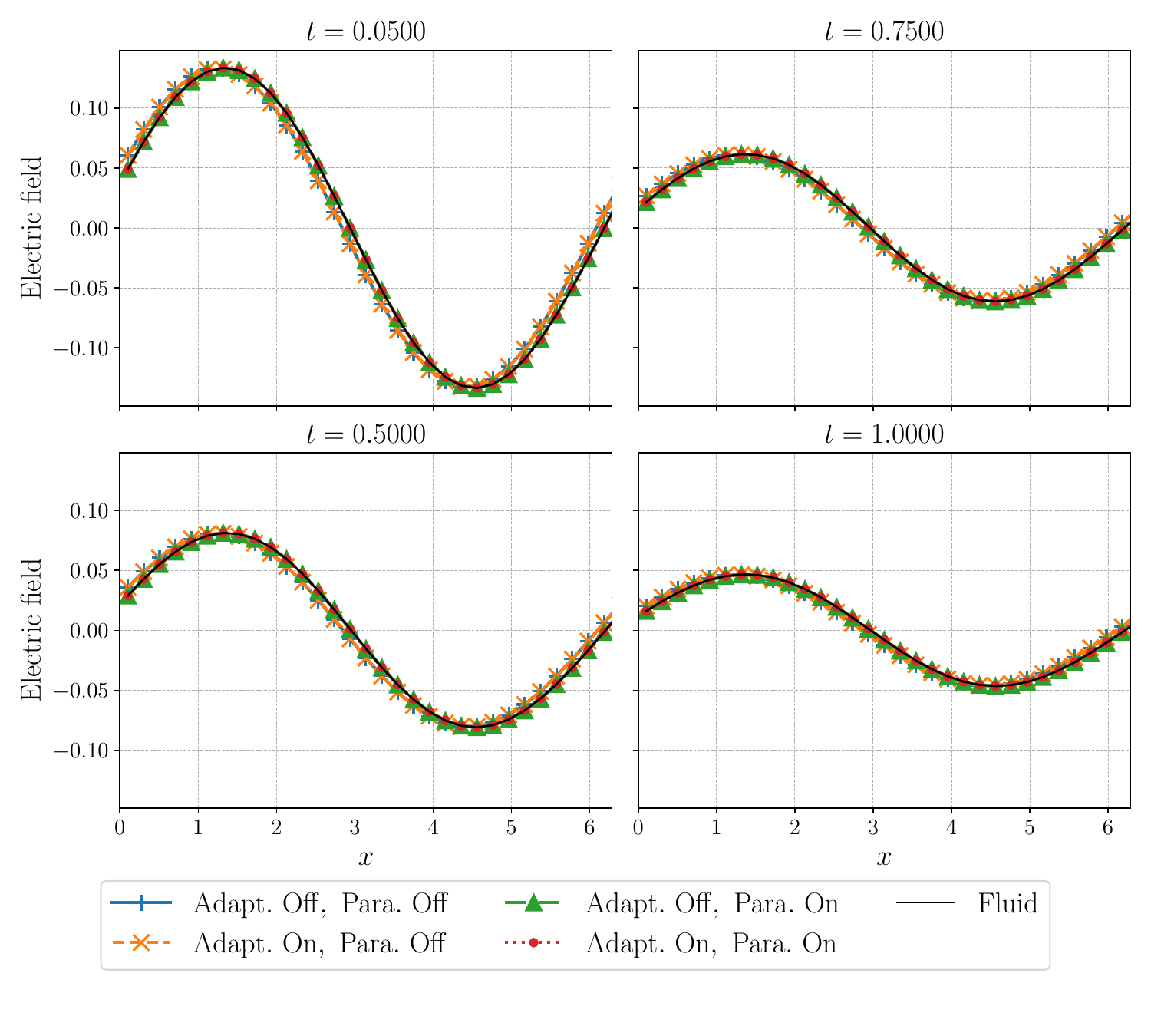}
  \caption{Electric field snapshots for $\ep=10^{-4}$ obtained with Algorithm~\ref{algo:paraMS}.}
  \label{fig:snapshotsEfield0001}
\end{figure}

\subsection{Convergence}
We investigate the convergence of Algorithm~\ref{algo:paraMS}. Figure~\ref{fig:ConvAlgo} shows the evolution of the error between two consecutive parareal iterations, for $\ep$ ranging from $10^{-7}$ to $1$ with and without domain adaptation, at final time $T_f=1.0$. The algorithm converges up to machine accuracy in all cases, with iteration counts strongly dependent on $\ep$: fewer than five iterations suffice for small $\ep$, while about twenty are required when $\ep \sim 1$. This behavior reflects that, as $\ep \to 0$, the fluid predictor provides a more accurate description of the system. Since the limiting system is parabolic, convergence is further improved, consistent with the known difficulty of the parareal method in transport-dominated regimes.

Dynamic domain adaptation further accelerates convergence. Once activated, the predictor coincides with the fine correction, reducing the number of parareal iterations. This improvement must be weighed against the computational trade-offs of looser coupling thresholds and the accuracy loss inherent in adopting the fluid approximation.

Figure~\ref{fig:LinfError} reports the $L^\infty$ error over time between the classical approximation (without parareal integration or domain adaptation) and the multiscale methods. Without parareal integration, the error remains small, since domain adaptation is only triggered in fluid regions where it provides a good approximation. Note that in the case $\ep=0.5$, domain adaptation does not trigger in short times and the error is exatly $0$. With parareal integration, a maximum error of order $10^{-2}$ is observed in the kinetic regime $\ep=0.5$, reflecting that the fluid predictor is too crude and even a second-order reconstruction of the perturbation cannot fully compensate. Nevertheless, second-order reconstruction yields the smallest error compared to lower-order variants. In the fluid regime $\ep=10^{-4}$, the error relative to the classical scheme drops to about $10^{-5}$. The sequencial in time hybrid method yields a smaller error, of the order of $10^{-6}$. Moreover, as the long-time behavior of \eqref{eq:kinetic} is relaxation to a constant state determined by the torus length and the initial mass~\cite{DolbeaultMouhotSchmeiser2015}, the error naturally vanishes as $t \to \infty$. Lastly, since the scheme \eqref{eq:SMM} is AP, in the limit $\ep\to0$, it captures the asymptotic limit which is the predictor of the parareal algorithm. As a consequence, the errors of all the methods studied in Figure~\ref{fig:LinfError} collapse to $0$ as kinetic and fluid schemes coincide.

It should be noted that the errors in Figure~\ref{fig:LinfError} are of a different nature from those in Figure~\ref{fig:ConvAlgo}. The latter illustrates the convergence of Algorithm~\ref{algo:paraMS}, without indicating the solution it converges to. A detailed analysis of this aspect lies beyond the scope of this work and would require a careful study of the multiscale operators’ properties, as performed in~\cite{SamaeySlawig2023}.

\begin{figure}
  \includegraphics[width=.49\linewidth]{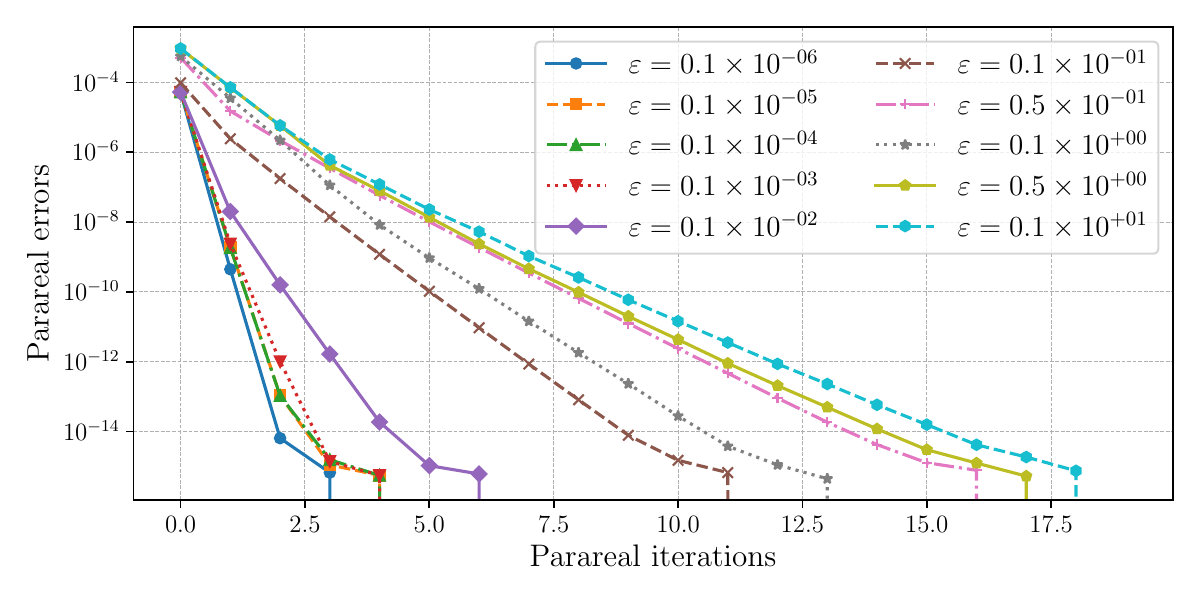}
  \includegraphics[width=.49\linewidth]{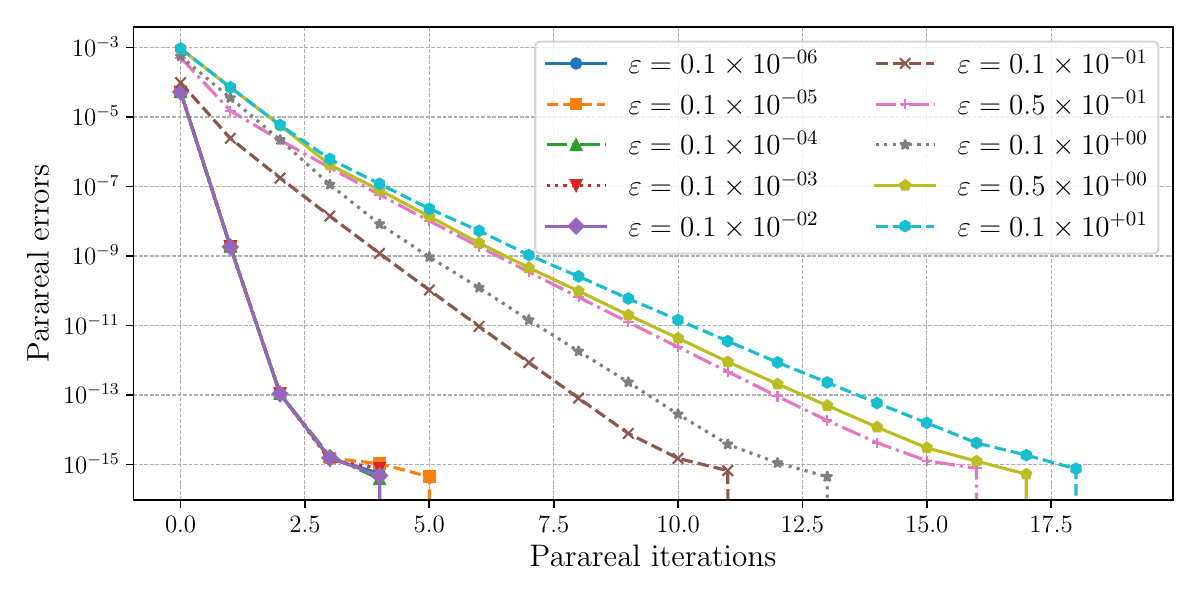}
  \caption{Convergence of Algorithm~\ref{algo:paraMS} for different values of $\ep$: without domain adaptation (left) and with domain adaptation (right).}
  \label{fig:ConvAlgo}
\end{figure}

\begin{figure}
  \includegraphics[width=.49\linewidth]{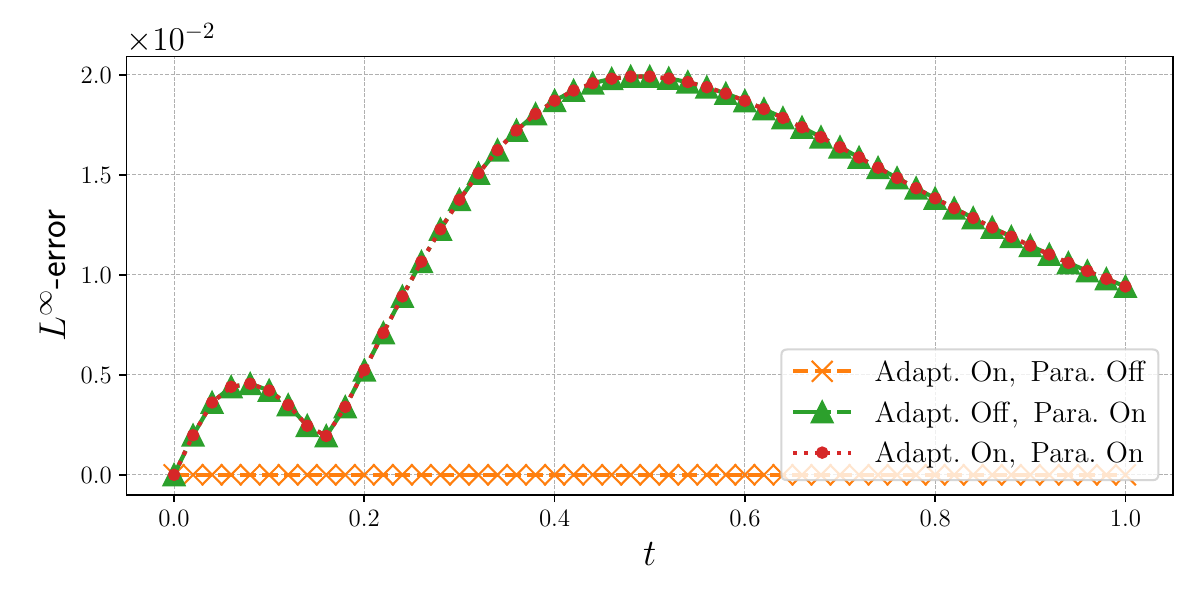}
  \includegraphics[width=.49\linewidth]{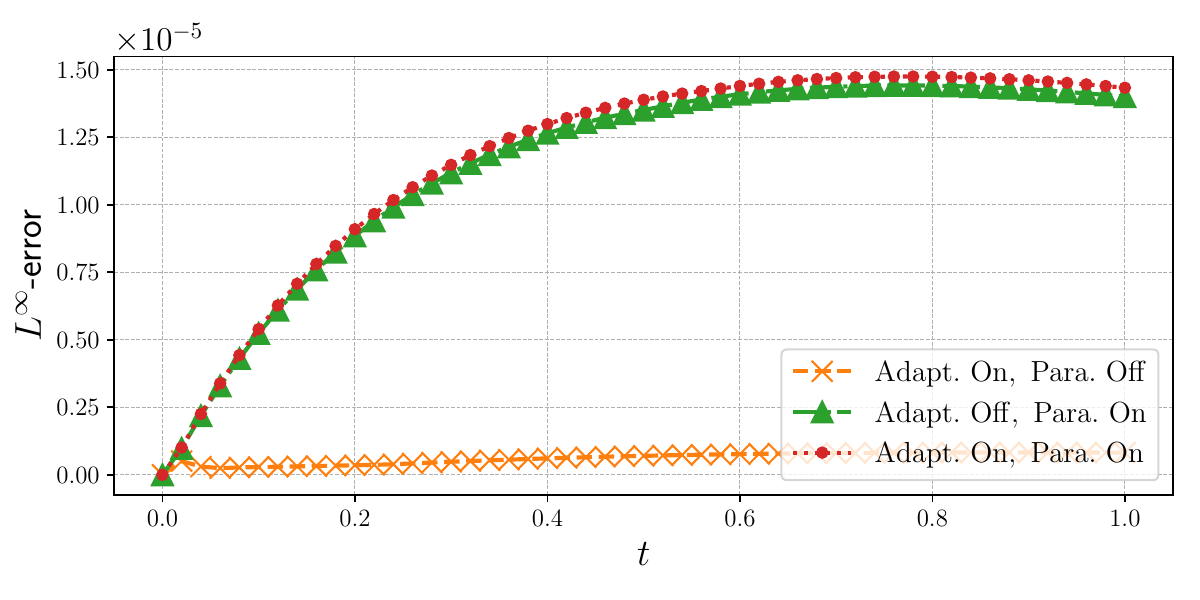}
  \caption{$L^\infty$-error of Algorithm~\ref{algo:paraMS} with respect to the classical scheme without parareal integration or domain adaptation, for $\ep=0.5$ (left) and $\ep=10^{-4}$ (right).}
  \label{fig:LinfError}
\end{figure}

\subsection{Performances}
To assess the performance of the fully hybrid method, we set $N_x = 32$, $N_{v_x} = 32$, $N_{v_y} = N_{v_z} = 16$ and $N_g=200$. In addition, the error threshold in the parareal algorithm is set to $10^{-10}$ and the coupling parameters are again set to $\eta_0 = \delta_0 = 10^{-5}$. The number of OpenMP threads is fixed to $48$ on the architecture described in Table~\ref{tab:architecture}. The corresponding runtimes and speedups are reported in Table~\ref{tab:runtimes-speedups}.

The results show that Algorithm~\ref{algo:paraMS} consistently outperforms the standard kinetic scheme for all values of $\ep$. This demonstrates the ability of the multiscale method to exploit parallel-in-time acceleration in kinetic regimes while benefiting from domain adaptation in transient and near-fluid regimes. Moreover, the method remains stable across the full range of $\ep$, confirming its AP property.

Without domain adaptation, for $\ep=1$ we observe a speedup of about $2$, which increases to $3$ at $\ep=10^{-2}$ and $7$ at $\ep=10^{-4}$. This first shows that the multiscale parareal algorithm applied to \eqref{eq:kinetic} can already reduce the computational cost of the time integration.

Enabling the domain adaptation, which mainly occurs in near-fluid regimes or in long time \cite{Laidin2023}, one observes that it indeed helps to further reduce the computational time with a speedup factor that reaches up to $73$ at $\ep=10^{-4}$. In the case $\ep = 10^{-3}$ (italicized values), a numerical interplay was observed between the two hybrid methods, and Algorithm~\ref{algo:paraMS} converged only to a successive error of $10^{-5}$, after five iterations. The runtime corresponding to these five iterations was used to compute the speedup.

Several remarks must be made regarding the performance of the algorithm. First, even with a relatively small number of grid points in each variable, significant speedups are already observed. These would naturally become even more significant as the cost of the kinetic solver increases drastically by refining the mesh. Secondly, the performance of the domain adaptation is directly related to the coupling thresholds $\eta_0$ and $\delta_0$ that are currently chosen empirically.

\begin{table}
  \centering
  \renewcommand{\arraystretch}{1.2}
  \begin{tabular}{c|c|c|c|c|c}
    & \multicolumn{4}{c|}{Runtimes (s)} & \\
    & \multicolumn{4}{c|}{\small \textbf{Speedups}} & \\
    \cline{2-5}
    & \multicolumn{2}{c|}{Para.~Off} & \multicolumn{2}{c|}{Para.~On} & \\
    \cline{2-5}
    $\ep$
      & Adapt.~Off & Adapt.~On
      & Adapt.~Off & Adapt.~On
      & Fluid \\
    \hline
    $1.0$ &
      \begin{tabular}{c}96.1\\\hline \textbf{1.00}\end{tabular} &
      \begin{tabular}{c}104.6\\\hline \textbf{0.92}\end{tabular} &
      \begin{tabular}{c}42.38\\\hline \textbf{2.27}\end{tabular} &
      \begin{tabular}{c}44.5\\\hline \textbf{2.16}\end{tabular} &
      \begin{tabular}{c}0.25\\\hline \textbf{384}\end{tabular} \\
    \hline
    $10^{-1}$ &
      \begin{tabular}{c}96.1\\\hline \textbf{1.00}\end{tabular} &
      \begin{tabular}{c}104.2\\\hline \textbf{0.92}\end{tabular} &
      \begin{tabular}{c}35.2\\\hline \textbf{2.73}\end{tabular} &
      \begin{tabular}{c}37.5\\\hline \textbf{2.56}\end{tabular} &
      \begin{tabular}{c}0.25\\\hline \textbf{384}\end{tabular} \\
    \hline
    $10^{-2}$ &
      \begin{tabular}{c}96.1\\\hline \textbf{1.00}\end{tabular} &
      \begin{tabular}{c}99.5\\\hline \textbf{0.97}\end{tabular} &
      \begin{tabular}{c}31.1\\\hline \textbf{3.09}\end{tabular} &
      \begin{tabular}{c}32.5\\\hline \textbf{2.96}\end{tabular} &
      \begin{tabular}{c}0.25\\\hline \textbf{384}\end{tabular} \\
    \hline
    $10^{-3}$ &
      \begin{tabular}{c}96.1\\\hline \textbf{1.00}\end{tabular} &
      \begin{tabular}{c}20.6\\\hline \textbf{4.67}\end{tabular} &
      \begin{tabular}{c}17.0\\\hline \textbf{5.65}\end{tabular} &
      \begin{tabular}{c}\textit{4.0}\\\hline \textbf{\textit{24.02}}\end{tabular} &
      \begin{tabular}{c}0.25\\\hline \textbf{384}\end{tabular} \\
    \hline
    $10^{-4}$ &
      \begin{tabular}{c}96.1\\\hline \textbf{1.00}\end{tabular} &
      \begin{tabular}{c}0.19\\\hline \textbf{506}\end{tabular} &
      \begin{tabular}{c}13.0\\\hline \textbf{7.39}\end{tabular} &
      \begin{tabular}{c}1.32\\\hline \textbf{72.80}\end{tabular} &
      \begin{tabular}{c}0.25\\\hline \textbf{384}\end{tabular} \\
    \hline
  \end{tabular}
  \caption{Runtimes and speedups of Algorithm~\ref{algo:paraMS} for different $\ep$.}
  \label{tab:runtimes-speedups}
\end{table}

\subsection{Mass conservation}
A key feature of \eqref{eq:kinetic} is the conservation of moments. In the linear relaxation case, only the mass is conserved. Figure~\ref{fig:massVar} reports the evolution of the mass variation
\begin{equation}
\Delta m^n = \sum_{i\in\II} \bigl(\rho_i^n - \rho_i^0\bigr)\Dx.
\end{equation}
Consistently with Lemma~\ref{lem:massCons}, Algorithm~\ref{algo:paraMS} preserves mass up to machine precision, independently of the activation of the dynamic domain adaptation.

\begin{figure}
\includegraphics[width=.49\linewidth]{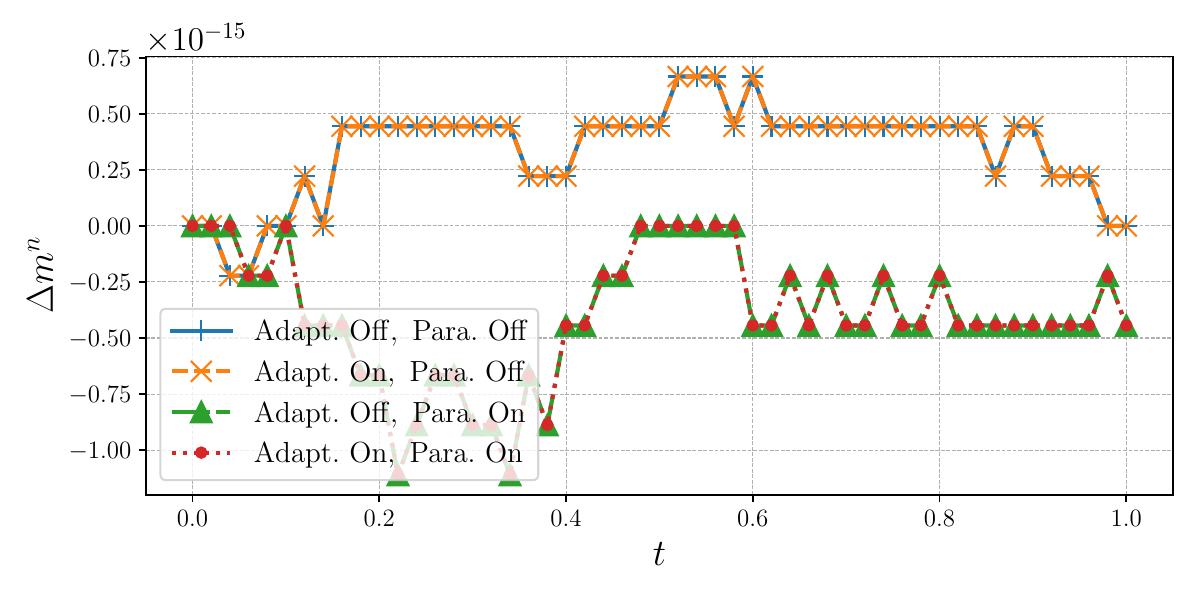}
\includegraphics[width=.49\linewidth]{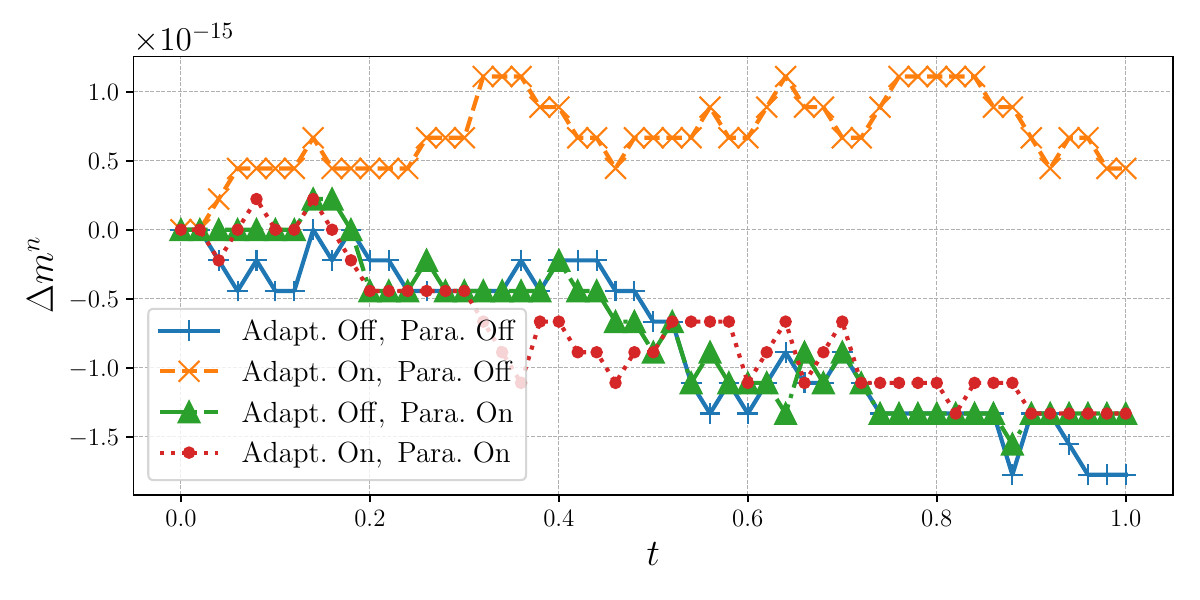}
\caption{Mass variation for Algorithm~\ref{algo:paraMS} with $\ep=0.5$ (left) and $\ep=10^{-4}$ (right).}
\label{fig:massVar}
\end{figure}

\section{Conclusion}
In this work we have presented a fully multiscale method for kinetic equations. The approach combines the advantages of AP schemes, dynamic domain adaptation, and parallel-in-time integration, leading to a significant reduction in numerical cost across all regimes of the scaling parameters. Apart from minimal numerical thresholds, the method accelerates computations automatically, without requiring user intervention or solver selection.

In addition, even larger computational gains can be expected in more complex settings, such as higher-dimensional problems in physical space, boundary conditions, or Boltzmann-type collision operators that significantly increase the cost of kinetic solvers, enhancing the need for accelerating methods.

The current implementation relies on a shared-memory framework. Possible extensions include implementing the method to distributed-memory architectures and parallelizing the hybrid solvers within the parareal algorithm. The method appears promising to extend towards more physically relevant models such as two species systems appearing in semiconductor modelling or plasma physics.

\section*{Acknowledgement}
The author would like to thank T. Rey and the Université Côte d'Azur for the computing resources used in this work.

The author has received funding from the European Research Council (ERC) under the European Union’s Horizon 2020 research and innovation program (grant agreement No 865711).

\bibliographystyle{acm}
\bibliography{STHBiblio}

\end{document}